\documentclass[a4paper,english]{article}
\usepackage[T1]{fontenc}
\usepackage{textcomp}
\usepackage{lmodern}
\usepackage[utf8]{inputenc}
\usepackage{babel}
\usepackage{mathtools}
\usepackage{amssymb}
\usepackage{amsthm}
\newtheorem{lemm}{Lemma}[section]
\newtheorem{coro}[lemm]{Corollary}
\newtheorem{prop}[lemm]{Proposition}
\newtheorem{theo}[lemm]{Theorem}
\newtheorem{defi}[lemm]{Definition}
\newtheorem{conj}[lemm]{Conjecture}
\DeclareMathOperator{\cayl}{Cay}
\DeclareMathOperator{\lcm}{lcm}
\DeclareMathOperator{\sq}{sq}
\DeclareMathOperator{\triangles}{N}
\DeclarePairedDelimiter{\abs}{\lvert}{\rvert}
\newcommand*{\DRR}{DRR}
\newcommand*{\GRR}{GRR}
\newcommand*{\orCayley}[2]{\overrightarrow\cayl( #1,#2)}
\newcommand*{\ORR}{ORR}
\newcommand*{\presentation}[2]{\langle#1\mid#2\rangle}
\newcommand*{\setst}[2]{\{#1 \mid #2\}}
\newcommand*{\Triangles}[2]{\triangles_3(#1,#2)} 
\newcommand*{\unCayley}[2]{\cayl( #1,#2)}
\newcommand*{\Z}{\mathbf{Z}}
\newcommand*{\N}{\mathbf{N}}
\usepackage[colorlinks,breaklinks,bookmarks,plainpages=false,unicode=true]{hyperref} 
\title{Cayley graphs with few automorphisms: the case of infinite groups}
\author{Paul-Henry Leemann, Mikael de la Salle}
\date{\today}
\hypersetup{pdftitle={Cayley graphs with few automorphisms}, pdfauthor={Paul-Henry Leemann, Mikael de la Salle}, pdfsubject={\GRR, \DRR{} and \ORR{} for infinite groups.},pdfkeywords={\GRR, \DRR, \ORR, Cayley graph, automorphisms of graphs, generalized dihedral group, generalized dicyclic group, regular automorphism group}}
\usepackage[colorinlistoftodos,textsize=footnotesize]{todonotes}

\begin{document}
\maketitle
%
%
%
%
%
%
%
%
%
%
\begin{abstract}
We characterize the finitely generated groups that admit a Cayley graph whose only automorphisms are the translations, confirming a conjecture by Watkins from 1976. The proof relies on random walk techniques. As a consequence, every finitely generated group admits a Cayley graph with countable automorphism group. We also treat the case of directed graphs.
\end{abstract}
\section{Introduction}
Given a group $G$ and a symmetric generating set $S \subseteq G\setminus\{1\}$, the Cayley graph $\unCayley{G}{S}$ is the simple unoriented graph with vertex set $G$ and an edge between $g$ and $h$ precisely when $g^{-1}h \in S$. By construction, the action by left-multiplication of $G$ on itself induces an action of the group on its Cayley graph, which is free and vertex-transitive.

We are here interested in the question to decide when $S$ can be chosen so that these are all the automorphisms of $\unCayley{G}{S}$, or equivalently the automorphism group of the graph acts freely and transitively on the vertex set. The main result of this paper is
\begin{theo}\label{thm:main} Every finitely generated group $G$ that is not virtually abelian admits a finite degree Cayley graph whose automorphism group is not larger than $G$ acting by left-translation.
\end{theo}
A Cayley graph of $G$ whose automorphism group acts freely on its vertex set is called a \emph{graphical regular representation}, or \GRR. 

The main result in \cite{LdlS2020} is similar to Theorem~\ref{thm:main}, but with the assumption of not being virtually abelian replaced by having an element of infinite (or sufficiently large) order and being non-abelian and non-generalized dicyclic (see section~\ref{sec:reminders} for definitions). Combining both results, we obtain that the infinite finitely generated groups that do not admit a \GRR{} are precisely the abelian and the generalized dicyclic groups.

Together with the results from \cite{MR0255446,MR0295919,MR0280416,MR0319804,MR0363998,MR0344157,MR0392660,MR0457275,HetzelThese,MR642043} that treated the case of finite groups\footnote{We refer to the introduction of \cite{LdlS2020} and the references therein for a more detailed exposition of the work on finite groups}, we obtain the following result. The equivalence between \ref{item:GRR} and \ref{item:Gnonexceptionnal} confirms Watkin's conjecture \cite{MR0422076}.
\begin{coro}\label{cor:main} For a finitely generated group $G$, the following are equivalent:
  \begin{enumerate}
  \item\label{item:GRR} $G$ admits a \GRR,
  \item\label{item:locFiniteGRR} $G$ admits a finite degree \GRR,
  \item\label{item:Gnonexceptionnal} $G$ does not belong to the following list:
  \begin{itemize}
  \item the non-trivial abelian groups different from $\Z/2\Z$ and $(\Z/2\Z)^n$ for $n \geq 5$,
  \item the generalized dicyclic groups,
  \item the following $10$ finite groups of cardinality at most
    $32$\footnote{with GAP IDs respectively [6,1], [8,3], [10,1], [12,3]
    [24,11], [32,26], [16,13], [16,6], [18,4], [27,3]. The first digit in the GAP ID is the order of the group, and the second is the label of the group in GAP's numbering of groups of that order. For example, the last group in the list is of order 27, and is the third in GAP's list of groups of order 27. It is also isomorphic to the free Burnside group $B(2,3)$.}: the dihedral
    groups of order $6, 8, 10$, the alternating group $A_4$, the
    products $Q_8\times \Z/3\Z$ and $Q_8\times \Z/4\Z$ (for $Q_8$ the
    quaternion group) and the four groups given by the presentations
\[ \presentation{a,b,c}{a^2=b^2=c^2=1, abc=bca = cab},\]
\[ \presentation{a,b}{a^8=b^2=1, b^{-1} a b =a^5},\]
\[ \presentation{a,b,c}{a^3=b^3=c^2=(ac)^2=(bc)^2=1, ab=ba},\]
\[ \presentation{a,b,c}{a^3=b^3=c^3=1, ac=ca, bc=cb,b^{-1} a b = ac}.\]

  \end{itemize}
  \end{enumerate}
\end{coro}

\subsection{Consequences}
Another consequence of our result is the following fact, which was a motivation of the second-named author for \cite{LdlS2020} and the present work, see \cite{delaSalleTessera}. This solves a conjecture raised in \cite{delaSalleTessera}.
\begin{coro}\label{cor:discrete_automorphism_group} Every finitely generated group admits a  finite degree Cayley graph whose automorphism group is countable.
\end{coro}
Recall the classical fact that the topology of pointwise convergence on vertices turns the automorphism group of a finite degree graph into a locally compact metrizable group: a sequence $(\phi_n)_n$ of automorphisms converges to $\phi$ if and only if for every $x$, $\phi_n(x) = \phi(x)$ for all but finitely many $n$. For this topology the stabilizer of a vertex is a compact subgroup, and therefore either finite or uncountable. Therefore Corollary \eqref{cor:discrete_automorphism_group} can be equivalently phrased as \emph{Every finitely generated group admits a finite degree Cayley graph whose automorphism group is discrete} (equivalently \emph{has finite stabilizers}).

Corollary \eqref{cor:discrete_automorphism_group} has interesting graph-theoretical consequences, that we now explain. Following \cite{MR3156647,MR3658820,delaSalleTessera}, given two graphs $X$ and $Y$ and a positive integer~$R$, we say that $Y$ is \emph{$R$-locally $X$} if every ball of radius $R$ in $Y$ appears as a ball of radius $R$ in $X$. A graph $X$ is \emph{local to global rigid} (LG-rigid) if there is $R>0$ such that any graph that is $R$-locally $X$ is covered by $X$. 
\begin{coro}\label{cor:LGrigid} Every finitely presented group admits a finite degree Cayley graph that is $LG$-rigid.
\end{coro}
The existence of a $LG$-rigid finite degree Cayley graph is actually equivalent to finite presentability, as a non finitely presented group cannot admit LG-rigid Cayley graphs, see \cite{delaSalleTessera}.

\subsection{About the proof}
Theorem~\ref{thm:main} is only new for groups with bounded exponents: the case of groups with elements of infinite (or arbitrarily large) order was covered in \cite{LdlS2020}. We have therefore concentrated our efforts for finitely generated torsion groups, but the proof that we finally managed to obtain turned out to apply without much more efforts in the generality of Theorem~\ref{thm:main}.

The proof relies partly on the results from \cite{LdlS2020}, and partly on new ideas involving random walks on groups. In particular we use some recent results by Tointon \cite{tointon} on the probability that two independent realizations of the random walk commute.
Necessary background of group theory is presented in Section~\ref{sec:reminders}.
We discuss the needed contributions from \cite{LdlS2020} in Section~\ref{sec:few_automorphisms:past}. The new aspects, including random walk reminders are presented in Section~\ref{sec:random_walks}. In the small Section~\ref{sec:conclusion}, we deduce the main theorem and its corollaries, and then in Section~\ref{sec:conjecture} we discuss a conjecture that would significantly simplify our proofs.
Finally, in Section~\ref{sec:directed} we briefly discuss some directed variants of the GRR problem.

\paragraph{Acknowledgements:} The first author was supported by NSF Grant No. $200021\_188578$.
The second author's research was supported by the ANR projects AGIRA (ANR-16-CE40-0022) and ANCG (ANR-19-CE40-0002-01).
Part of this work was performed within the framework of the LABEX MILYON (ANR-10-LABX-0070) of Universit\'e de Lyon, within the program ``Investissements d'Avenir'' (ANR-11-IDEX-0007) operated by the French National Research Agency (ANR). Both authors thank the anonymous referes for their useful comments and suggestions.

\section{Group theory background}\label{sec:reminders}

This short section sets the group theoretical notation and terminology and contains some standard group theory results that will be used later. It can be safely skipped by most readers.

We start with a definition used in the introduction.

\begin{defi}\label{def:gen_dicyclic}A \emph{generalized dicyclic group} is a non-abelian group with an abelian subgroup $A$ of index $2$ and an element $x$ not in $A$ such that $x^4=1$ and $xax^{-1} = a^{-1}$ for all $a \in A$.
\end{defi}

A group is said to be of \emph{exponent $n$} if every element satisfies $g^n=1$, and of \emph{bounded exponent} if it is of exponent $n$ for some integer $n$. It is known that finitely generated groups of exponent $n$ are finite if $n \in \{1,2,3,4,6\}$, but can be infinite for large $n$, see  \cite[14.2]{MR1357169}. 

We shall denote the index of a subgroup $H$ of $G$ by $|G:H|$.

If $A$ is a subset of a group $G$, the \emph{centralizer} of $A$ in $G$ is the subgroup denoted $C_G(A)$ of elements of $G$ commuting with every element of $A$. The centralizer of $G$ in $G$ is the \emph{center} $Z(G)$ of $G$. If $g \in G$, we write $C_G(s)$ for $C_G(\{s\})$.

In the following, by \emph{group property} we mean of property $P$ that a group $G$ can have (``$G$ is $P$'')  or not (``$G$ is not $P$'').

\begin{defi} If $P$ and $Q$ are two group properties, we say that a group $G$ is:

  \begin{itemize}
  \item $P$-by-$Q$ if it admits a normal subgroup $N$ that is $P$ such that the quotient group $G/N$ is $Q$,
  \item locally $P$ if every finitely generated subgroup of $G$ is $P$,
  \item virtually $P$ if $G$ admits a finite index subgroup $H$ that is $P$.
  \end{itemize}
\end{defi}

Observe that a locally finite group $G$ is finitely generated if and only if it is finite.

A group $G$ is \emph{$2$-nilpotent} (or nilpotent of class $\leq 2$) if it is an extension $1\to N \to G \to K \to 1$  of abelian groups $N,K$ with $N \leq Z(G)$

We will use later the following elementary fact:
\begin{lemm}\label{lemma:almost_central_extension_of_almost_abelian} If $1\to N \to G \to K \to 1$ is an extension with $K$ virtually abelian and $|G:C_G(N)|<\infty$, then $G$ is virtually $2$-nilpotent.
\end{lemm}
\begin{proof} If $K_1<K$ is abelian with finite index, then the intersection of the preimage of $K_1$ in $G$ with the centralizer $C_G(N)$ of $N$ in $G$ is the intersection of two finite index subgroups of $G$ and therefore has finite index in $G$. It is clearly $2$-nilpotent.
\end{proof}
We also need the following folklore variant:
\begin{lemm}\label{lemma:finite_by_abelian} If a finitely generated group is finite-by-(virtually abelian), then it is virtually abelian.
\end{lemm}
\begin{proof} A finite-by-(virtually abelian) group $G_0$ contains a finite index finite-by-abelian group $G$ (the preimage in $G_0$ of the finite index abelian subgroup in the quotient).  So let $1\to N \to G \to K \to 1$ with $N$ finite and $K$ abelian. Let $g \in G$. Since $K$ is abelian, the conjugacy class of $g$ is contained in $gN$ and is thus finite; equivalently $C_G(g)$ has finite index in $G$. If $G_0$ is finitely generated, then so is $G$. If $S$ is a finite generating set of $G$, the center of $G$ is equal to $\cap_{g \in S} C_G(g)$ and therefore is an abelian subgroup of finite index in $G$. This proves that $G$, and hence $G_0$, is virtually abelian.
\end{proof}
We will make use of the following lemma which is due to Dicman, see for example \cite[14.5.7]{MR1357169}.
\begin{lemm}[Dicman's Lemma]\label{lem:Dicman}
Let $G$ be a group and $X\subseteq G$ be a finite subset that is invariant by conjugation by elements of $G$ and such that every $g\in X$ is of finite order.
Then the normal subgroup $\langle X\rangle^G$ is finite.
\end{lemm}

\section{Constructing Cayley graphs with few automorphisms}\label{sec:few_automorphisms:past}

We follow the same general strategy for constructing Cayley graphs with few automorphisms as the one initiated in \cite{delaSalleTessera} and later developed in \cite{LdlS2020}. There are two independent steps.

The first step consists in finding a finite symmetric generating set $S_1$ of $G$ in which the only knowledge of the \emph{colour} $\{g^{-1}h,h^{-1}g\} \subseteq S_1$ of every edge $\{g,h\}$ allows us to reconstruct the orientation of every edge. To say it in formulas, following \cite{LdlS2020} let us say that the pair $(G,S_1)$ is \emph{orientation-rigid} if the only permutations $\phi$ of $G$ such that $\phi(gs) \in \{\phi(g)s,\phi(g)s^{-1}\}$ for every $g \in G$ and $s \in S_1$ are the left-translations by elements of $G$.

The second step is, given a finite symmetric generating set $S_1 \subseteq G$, to find another finite symmetric generating set $S_1\subseteq S_2$ such that every automorphism $\phi$ of $\unCayley{G}{S_2}$ induces a colour-preserving automorphism of $\unCayley{G}{S_1}$, that is satisfies $\phi(gs) \in \{\phi(g)s,\phi(g)s^{-1}\}$ for every $g \in G$ and $s \in S_1$.

Clearly, if we are able to perform both steps and we apply the second step for $S_1$ that is orientation-rigid, then we obtain a Cayley graph $\unCayley{G}{S_2}$ in which the only automorphisms are the translations.

For the first step, there is nothing new to do, as we know exactly which groups admit an orientation-rigid pair $(G,S_1)$ with $S_1$ generating. The following is a portion of \cite[Theorem 7]{LdlS2020}.
\begin{theo}\label{orientation-rigid} Let $G$ be a finitely generated group that is neither abelian [with an element of order greater than $2$] nor generalized dicyclic, and $S_0 \subseteq G\setminus\{1\}$ be a finite symmetric generating set. Then $(G,S_1)$ is orientation-rigid where $S_1 = (S_0 \cup S_0^2 \cup S_0^3)\setminus\{1\}$.
\end{theo}

So all the new work lies in the second step. As in \cite{delaSalleTessera,LdlS2020}, the main tool to recognize the colour of the edges is by \emph{counting triangles}. If $S$ is a finite symmetric subset of a group $G$, we denote by $\Triangles{s}{S}$ the number of triangles in the Cayley graph $\unCayley{G}{S}$ containing both vertices $1$ and $s$:
\[ \Triangles{s}{S} \coloneqq \abs{S\cap sS} \cdot \mathbf1_{s \in S}.\]
The relevance of this is the following easy observation that automorphisms preserve the number of triangles of a given edge: if $\phi$ is an automorphism of $\unCayley{G}{s}$, then for every $g,h \in G$,
\[  \Triangles{g^{-1} h}{S} = \Triangles{\phi(g)^{-1} \phi(h)}{S}.\]
Clearly, $\Triangles{s}{S} = \Triangles{s^{-1}}{S}$, so counting triangles can never do more than recover the colour of edges. But if $s \in S$ is such that the only elements $t$ of $S$ such that $\Triangles{s}{S} = \Triangles{t}{S}$ are $s$ and $s^{-1}$, then the automorphism group of $\unCayley{G}{S}$ preserves the colour of every edge corresponding to $s$.
Our main new technical result is the following.
\begin{prop}\label{Prop:9.3}
Let $G$ be a finitely generated group that is not virtually abelian and $S\subset G\setminus\{1\}$ be a finite generating set.
Then there exists a finite symmetric generating set $S\subseteq \tilde S\subset G\setminus\{1\}$ of size bounded by $2|S|(|S|+14)$ such that for all $s\in S$ and $t\in\tilde S$, if $\Triangles{s}{\tilde S}=\Triangles{t}{\tilde S}$ then $t=s$ or $t=s^{-1}$.

Moreover, $\tilde S \setminus S$ does not have elements of order $2$.
\end{prop}
The next section is devoted to the proof of the above Proposition.
\section{Squares and random walks}\label{sec:random_walks}
The aim of this section is to prove Proposition~\ref{Prop:9.3}. We start with some discussions on the \emph{square map} $\sq\colon G\to G$ defined by $\sq(g)=g^2$, which will play an important role in our work. The main result of this section is Proposition~\ref{Prop:ultimate}, that will be proved using random walks.
\subsection{On the squares in finitely generated groups}
In our previous work, \cite{LdlS2020}, we restricted our attention on groups $G$ with an element of ``big'' (possibly infinite) order.
In other words, we asked $G$ to have a ``big'' cyclic subgroup $C$.
The main advantage of this hypothesis is the fact that in cyclic groups $\sq^{-1}(g)$ consist of at most $2$ elements and therefore, for any finite subset $F$ of $G$ the set $\sq^{-1}(F)\cap C$ is finite of size at most $2\abs{F}$.
In order to generalize results from \cite{LdlS2020} to arbitrary infinite finitely generated groups, we first need to establish some facts on the map $\sq$ and on its fibers.
We begin our analysis of the map $\sq$ by showing that infinite finitely generated groups contains infinitely many squares.

As a first consequence of Dicman's Lemma~\ref{lem:Dicman}, we obtain the following 
\begin{lemm}\label{Lemma:InfiniteElementsOrderNot4bis}
  Let $n$ be a fixed integer in $\{1,2,3,4,6\}$. If $G$ is an infinite finitely generated group, then $\setst{g^2}{g \in G,g^n\neq 1}$ is infinite.
\end{lemm}
\begin{proof}
By contradiction, suppose that the set $X=\setst{g^2}{g \in G, g^n\neq 1}$ is finite.
This set is invariant by conjugation and, since it is finite, contains only elements of finite order.
Hence by Dicman's Lemma, it generates a finite normal subgroup~$N$.
But then $G/N$ is an infinite finitely generated group in which every element satisfies $g^n=1$ or $g^2=1$, so is of exponent $n'=n$ if $n$ is even and $n'=2n$ is $n$ is odd. In both cases, $n' \in \{1,2,3,4,6\}$, so the free Burnside group $B(m,n')$ being finite (see for example \cite[14.2]{MR1357169}), the group $G/N$ is finite, which is the desired contradiction.
\end{proof}

\begin{coro}\label{Cor:Inftysquares}
Let $G$ be a finitely generated group.
Then $G$ is infinite if and only if $\sq(G)$ is infinite.
\end{coro}

The following elementary lemma will be useful.
\begin{lemm}\label{lemma:squares_and_centralizers}
Let $a,b,s$ be elements of a group. If $a^2=b^2$ and $(sa)^2=(sb)^2$, then $[ab^{-1},s]=1$.
\end{lemm}
\begin{proof}We have
\[ ab^{-1} s= a b^{-2} s^{-1} (sb)^2b^{-1}  = a a^{-2} s^{-1} (sa)^2 b^{-1} = s a b^{-1}.\qed\]
\let\qed\relax
\end{proof}

We now state our main new contribution, which can be seen as a much stronger form of Corollary~\ref{Cor:Inftysquares}, for finitely generated groups that are not virtually abelian. 
The next result roughly asserts that, in such a group, not even the half of the elements can have finitely many squares.
\begin{prop}\label{Prop:ultimate} Let $G$ be a finitely generated group that is not virtually abelian. Then for every $s \in G$ and $F \subset G$ finite there are infinitely many $g \in G \setminus (\sq^{-1}(F) \cup s \sq^{-1}(F))$ such that
\[\begin{cases}
g^{-1} s g \notin F&\textnormal{if $C_G(s)$ is locally finite}\\
g^{-1} s g = s&\textnormal{otherwise.}
\end{cases}\]
\end{prop}
The fact that $G$ is not virtually abelian is essential in the proposition. For example, the result is not true for the infinite dihedral group $G=\Z/2\Z \ast \Z/2\Z$. Indeed, if $s$ is one of the generators of the free factors, then $G = \sq^{-1}(1) \cup s \sq^{-1}(1)$. Another example is provided by a generalized dicyclic group $G$: if $x \in G$ is an element of order $4$ such that the conjugation by $x$ induces the inverse on an index $2$ abelian subgroup $A$, then $G = \sq^{-1}(x^2) \cup x \sq^{-1}(x^2)$. This example shows that $F$ does not necessarily contain the identity. 

The proof of this proposition will rely on the following lemma, which strengthens Neumann's lemma \cite[Lemma 4.1]{MR62122}.
\begin{lemm}\label{lemma:Mikael2} Let $G$ be a finitely generated group, $s \in G$ and $H_1,\dots,H_m$ be subgroups of $G$ and $a_1,\dots,a_m \in G$. Assume that
  \begin{equation}\label{eq:main_equation} G = \sq^{-1}(1) \cup s\sq^{-1}(1)\cup a_1 H_1\cup \dots \cup a_m H_m.\end{equation}

  Let $\alpha:= \sum_{i=1}^m \frac{1}{|G:H_i|}$ denote the sum of the inverses of the indices of $H_i$, with the standard convention $\frac 1 \infty = 0$.

  \begin{enumerate}
    \item\label{item:s2_finite_conjugacy_class} If $\alpha < \frac{3-\sqrt{5}}{4} \simeq 0.19$, then either $G$ is virtually abelian, or the conjugacy class of $s^2$ is finite of cardinality less than $\frac{4}{(3-\sqrt{5}-4\alpha)^2}$.
    \item\label{item:involution_implies_virt_ab} If $s$ is an involution, and $\alpha < \frac{(1-\alpha)^3}{24}$ (eg $\alpha \leq 0.035$), then $G$ is virtually abelian.
  \end{enumerate}

\end{lemm}
The proof of the proposition will only use the lemma in the case when the $H_i$ all have infinite indices, \emph{i.e.} when $\alpha=0$, but we find this quantitative form amusing.

We will prove the above lemma in the next subsection. Let us first explain how the proposition follows.

 \begin{proof}[Proof of Proposition~\ref{Prop:ultimate}]
    Let $G$ be a finitely generated group. Suppose that there exists $s \in G$ and $F \subset G$ finite such that there are only finitely many $g \in G \setminus (\sq^{-1}(F) \cup s \sq^{-1}(F))$ such that
\[\begin{cases}
g^{-1} s g \notin F&\textnormal{if $C_G(s)$ is locally finite}\\
g^{-1} s g = s&\textnormal{otherwise.}
\end{cases}\]
We will show that such a group is virtually abelian.
We will do that in several steps. First we will show that $C_G(s)$ is locally finite. Then by multiple reductions we will prove that $G$ is virtually $2$-nilpotent and finally that $G$ is in fact virtually abelian. 
    
\paragraph{$C_G(s)$ is locally finite.} Suppose that $C_G(s)$ is not locally finite. By Corollary~\ref{Cor:Inftysquares}, $C_G(s)$ has infinitely many squares. This implies that  $C_G(s)\setminus(\sq^{-1}(F) \cup s \sq^{-1}(F))$ is infinite for any finite subset $F \subset G$, as otherwise this would imply that for all but finitely many $g \in C_G(s)$ satisfy $g^2 \in F$ or $g^2 = s^2(s^{-1} g)^2 \in s^2 F$, \emph{i.e.} $C_G(s)$ has finitely many squares.
    
\paragraph{}We know that $C_G(s)$ is locally finite. 
By assumption, there exist finite subsets $E,F \subset G$ such that
\begin{equation}\label{eq:absurd_decomposition_with_E}G = \sq^{-1}(F) \cup s \sq^{-1}(F) \cup C_G(s) E.
\end{equation}

We will now prove that $G$ is virtually $2$-nilpotent.
\paragraph{Reduction to $F=\{1\}$.} Let $F_1 \subseteq F$ be the subset of elements with finite conjugacy class (the intersection of $F$ with the FC-center of $G$).

The set of $g \in G$ such that $g (F \setminus F_1) g^{-1} \cap F \neq \emptyset$ is a finite union of cosets of the groups $C_G(f)$ for $f \in F \setminus F_1$. So it is a finite union of cosets of infinite-index subgroups. By Neumann's lemma \cite[Lemma 4.1]{MR62122}, this finite union is a strict subset of $G$, and there is $g_0 \in G$ such that $g_0 (F\setminus F_1) g_0^{-1} \cap F = \emptyset$. For such a $g_0$ and for every $h \in \sq^{-1}(F\setminus F_1)$ we have $g_0hg_0^{-1} \notin \sq^{-1}(F)$, so $g_0hg_0^{-1} \in s \sq^{-1}(F) \cup C_G(s) E$.
This implies that, on $g_0 \sq^{-1}(F\setminus F_1) g_0^{-1} \setminus C_G(s) E$, both maps $g \mapsto g^2$ and $g\mapsto (s^{-1}g)^2$ take finitely many values, so by Lemma \ref{lemma:squares_and_centralizers} we obtain
\[ g_0 \sq^{-1}(F\setminus F_1) g_0^{-1} \setminus C_G(s) E \subseteq C_G(s) E_1\] for some finite set $E_1$, or equivalently
\[ \sq^{-1}(F\setminus F_1) \subseteq g_0^{-1} C_G( s) (E \cup E_1) g_0.\]
In particular, we deduce from \eqref{eq:absurd_decomposition_with_E} that
\[ G = \sq^{-1}(F_1) \cup s \sq^{-1}(F_1) \cup A C_G(s) B\]
for some finite subsets $A,B \subset G$.

Denote by $N$ the subgroup generated by the finite set $F_1^G:=\cup_{g \in G} gF_1 g^{-1}$.
Then $N$ is normal, and its centralizer in $G$, which is the intersection of the centralizers of $f$ for $f \in F_1^G$, is a finite intersection of finite index subgroups, so has finite index.

Let $G'=G/N$, and $s',A',B',H'$ be the images of $s,A,B,C_G(s)$ in $G'$ respectively. In the quotient, the previous equation becomes
\begin{equation}\label{eq:absurd_decomposition_in_G'}
G' = \sq^{-1}(1) \cup s' \sq^{-1}(1) \cup A' H' B'.
\end{equation}
By Lemma~\ref{lemma:almost_central_extension_of_almost_abelian}, either $G$ is virtually $2$-nilpotent and we are done, or $G'$ is not virtually abelian.
We can therefore suppose that $G'$ is not virtually abelian and hence infinite.

\paragraph{Reduction to ${s'}^2=1$.} Observe that $H'$, the image of the locally finite group $C_G(s)$ in the quotient $G/N$, is locally finite, so it cannot have finite index since $G'$ is finitely generated and infinite and so are its finite index subgroups. 
We deduce by \ref{item:s2_finite_conjugacy_class} in Lemma \ref{lemma:Mikael2} that ${s'}^2$ has finite conjugacy class. Moreover, $C_G(s)$ being locally finite, ${s'}^2$ has finite order. By Dicman's Lemma, the normal subgroup $M$ generated by ${s'}^2$ is finite.
Let $\tilde G\coloneqq G'/M$.
Then $\tilde G = \sq^{-1}(1) \cup \tilde s \sq^{-1}(1) \cup \tilde A \tilde H \tilde B$ with $\tilde s^2=1$, $\tilde H$ an infinite index subgroup and $\tilde A$ and $\tilde B$ finite subsets.

\paragraph{The group $G$ is virtually $2$-nilpotent.} By a direct application of \ref{item:involution_implies_virt_ab} in Lemma \ref{lemma:Mikael2}, we obtain that $\tilde G$ is virtually abelian.
This implies that $G'$ is finite-by-(virtually abelian) and, since it is finitely generated, virtually abelian (see Lemma~\ref{lemma:finite_by_abelian}). This is the desired contradiction.

\paragraph{The group $G$ is virtually abelian.}
We already know that $G$ is finitely generated and virtually $2$-nilpotent.
Let $H$ be a finite index subgroup of $G$ that is $2$-nilpotent.

By definition, the derived subgroup $\langle ghg^{-1}h^{-1}\mid g,h\in H\rangle$ of $H$ is contained in the center of $H$.
Since $H$ is nilpotent, we also know that the subset of torsion elements is a subgroup of $H$ (see for example \cite[5.2.7]{MR1357169}) and that all subgroups of $H$, and also of $G$, are finitely generated (see for example \cite[5.4.6]{MR1357169}). 
In particular, since $C_G(s)$ is locally finite, it is finite.
By \eqref{eq:absurd_decomposition_with_E}, this implies that there exists a finite $F'\subset G$ such that 
\[
G = \sq^{-1}(F') \cup s \sq^{-1}(F').
\]
We claim that there exists $n\in \N$ such that  $s g^{n} s^{-1} = g^{-n}$ for every $g$ in~$G$.
Indeed, let $n\coloneqq \lcm\setst{k}{1\leq k\leq 3\abs F}$.
Then if $g$ has order at most $3\abs {F'}$ we have $g^n=1$ and the desired identity holds.
On the other hand, let $g$ be of order at least $3\abs {F'}+1$.
For such a $g$, there exists at most $2\abs {F'}$ integers $1\leq k\leq 3\abs {F'}+1$ such that $g^{2k}$ is in ${F'}$.
Therefore, there is at least $\abs {F'}+1$ integers $1\leq k\leq 3\abs {F'}+1$ with $(sg^k)^2\in {F'}$. By the pigeonhole principle, we have $1\leq k\neq l\leq 3\abs {F'}+1$, which is less than the order of $g$, such that $(sg^k)^2=(sg^l)^2$ and hence $s g^{k-l} s^{-1} = g^{-(k-l)}$ and the desired identity holds.
For every $g,h \in G$, we have
\begin{align}\label{eq:commute_up_to_N}
    (g^n h^n)^{n} &= s  (g^n h^n)^{-n} s^{-1} = (s g^n s^{-1} s h^n s^{-1})^{-n}\notag\\& =(g^{-n} h^{-n})^{-n} =(h^n g^n)^{n}.
\end{align}

When $g,h \in H$, $a:=[g^n,h^n]$ belongs to the center of $H$, so \eqref{eq:commute_up_to_N} shows that $a^n=1$.
The subgroup $K<Z(H)$ of elements of the center that are of finite order is finite.
In the quotient $H/K$, the subgroup $H'$ generated by $\setst{g^n}{g \in H/K}$ is abelian. Moreover, the quotient $(H/K)/H'$ is a finitely generated nilpotent group of exponent $n$, so is finite (see for example \cite[5.2.18]{MR1357169}). This implies that $H$, and therefore also $G$, is virtually abelian.
 \end{proof}

\subsection{Random walks and proof of Lemma \ref{lemma:Mikael2}}\label{subsection:ProofOfLemma}

The proof of Lemma \ref{lemma:Mikael2} will use random walk techniques, and in particular the recent result of Tointon \cite{tointon} generalizing to infinite groups a  classical result by P.~Neumann \cite{MR1005821} roughly saying: \emph{a finite group in which the probability that two randomly chosen elements commute is large is almost abelian}. We fix a symmetric probability measure $\mu$ on $G$ whose support is finite, generates $G$ and contains the identity. In particular, in this subsection $G$ will always be a finitely generated group. Let $g_n$ and $g'_n$ be two independent realizations of the random walk on $G$ given by $\mu$, that is two independent random variables with distribution $\mu^{\ast n}$, the $n$-th convolution power of $\mu$. We will use two facts. The first is very easy\footnote{The law of a {\bf reversible} aperiodic transitive random walk on a set $V$ equidistributes if $V$ is finite, and converges $\sigma(\ell_1(V),c_0(V))$ to $0$ if $V$ is infinite.} and asserts that if $H$ is a subgroup of $G$ and $a \in G$, then
\begin{equation}\label{eq:proba_of_coset} \lim_n \mathbf{P}(g_n \in aH) = \frac{1}{|G:H|}.\end{equation}
In particular, if $H$ has infinite index, $\lim_n \mathbf{P}(g_n \in aH) = 0$. Actually, more is known: the above convergence is uniform in $a$ and $H$. In the vocabulary of \cite{tointon}, $\mu^{\ast n}$ \emph{measures indices uniformly}, see \cite[Theorem 1.11]{tointon}. To illustrate the power of random walks on groups, observe that \eqref{eq:proba_of_coset} allows to give a transparent proof (for finitely generated groups) of Neumann's lemma: if $G=a_1 H_1\cup \dots a_m H_m $ is a finite union of cosets of subgroups, then we have
\[ \sum_{i=1}^m \frac{1}{|G:H_i|} = \lim_n \sum_{i=1}^m \mathbf{P}(g_n \in a_i H_i) \geq \lim_n  \mathbf{P}( g_n \in \cup_i a_i H_i)=1.\]

The second fact we will use, \cite[Theorem 1.9]{tointon}, asserts that, whenever $G$ is not virtually abelian
\begin{equation}\label{eq:proba_of_commuting} \lim_n \mathbf{P}(g_n \textnormal{ and }g_n'\textnormal{ commute}) = 0.\end{equation}

We start with an easy consequence of \eqref{eq:proba_of_commuting}, that we will use in the proof. It is natural to expect that the result holds with $\frac{\sqrt{5}-1}{2}$ replaced by an arbitrary positive number, see Section~\ref{sec:conjecture}. 
\begin{lemm}\label{lemma:prob_of_involution} If $\liminf_n \mathbf{P}(g_n^2 = 1) > \frac{\sqrt{5}-1}{2}$, then $G$ is virtually abelian.
\end{lemm}
\begin{proof}
Denote $c_n \coloneqq \mathbf{P}(g_n^2 = 1)$. Observe that
\begin{align*}
\mathbf{P}(g_n^2 = {g'_n}^2=(g_ng'_n)^2=1) &= \mathbf{P}(g_n^2 = {g'_n}^2=1) - \mathbf{P}(g_n^2 = {g'_n}^2=1 \neq(g_ng'_n)^2)\\
    &\geq \mathbf{P}(g_n^2 = {g'_n}^2=1) - \mathbf{P}((g_ng'_n)^2\neq 1)\\
    & = c_n^2 -(1- c_{2n}).
\end{align*}
In the last line, we used that $g_ng'_n$ is distributed as $g_{2n}$.

So, if $c=\liminf_n c_n$, we obtain
\[ \liminf_n \mathbf{P}(g_n^2 = {g'_n}^2=(g_ng'_n)^2=1) \geq c^2 +c-1,\]
which is positive if and only if $c>\frac{\sqrt{5}-1}{2}$. To conclude using \eqref{eq:proba_of_commuting}, it remains to observe that $g_n^2 = {g'_n}^2=(g_ng'_n)^2=1$ implies that $g_n$ and ${g'_n}$ commute.
\end{proof}

We now proceed to prove Lemma \ref{lemma:Mikael2}.
\begin{proof}[Proof of \ref{item:s2_finite_conjugacy_class}. in Lemma \ref{lemma:Mikael2}] Let $G,s,m,H_i,a_i$ satisfying \eqref{eq:main_equation}.

Assume that $\alpha < \frac{3-\sqrt{5}}{4}$.

By \eqref{eq:main_equation}, for every $g \in G$ such that $g^2 \neq 1$, we have
\[ (sg)^2 = 1\textrm{ or }sg \in a_1 H_1\cup \dots \cup a_m H_m\]
and
\[ (s^{-1} g)^2 = 1 \textrm{ or }g \in a_1 H_1\cup \dots \cup a_m H_m.\]
In particular, if $g^2 \neq 1$ and $g \notin \{1,s^{-1}\} (a_1 H_1\cup \dots \cup a_m H_m)$, we have $(sg)^2=(s^{-1} g)^2 = 1$, which implies
\[g^{-1}s^2g=(g^{-1}s)(sg) = (s^{-1}g)(g^{-1}s^{-1}) = s^{-2}.\] 
Therefore we obtain
\[ \mathbf{P}(g_n^{-1} s^2 g_n=s^{-2}) \geq \mathbf{P}(g_n^2 \neq 1\textrm{ and }g_n \notin \{1,s^{-1}\} (a_1 H_1\cup \dots \cup a_m H_m)).\]
If $c:=\liminf_n \mathbf{P}(g_n^2 = 1) > \frac{\sqrt{5}-1}{2}$, we know by Lemma \ref{lemma:prob_of_involution} that $G$ is virtually abelian. So we can as well assume that $c \leq \frac{\sqrt{5}-1}{2}$.  By \eqref{eq:proba_of_coset} we can bound
\[\limsup_n \mathbf{P}(g_n^{-1} s^2 g_n=s^{-2}) \geq 1 - c - 2 \sum_{i=1}^m \frac{1}{|G:H_i|} \geq \frac{3-\sqrt{5}-4\alpha}{2}.\]
Observe that $\frac{3-\sqrt{5}-4\alpha}{2}>0$ by our assumption on $\alpha$.
Now, if $g_n^{-1} s^2 g_n=s^{-2}$ and ${g'_n}^{-1} s^2 {g'_n}=s^{-2}$, then $g_n{g'_n}$ and $s^2$ commute, or equivalently $g_n{g'_n} \in C_G(s^2)$. Therefore, since $g_n{g'_n}$ is distributed as $g_{2n}$, we obtain
\[\limsup_n \mathbf{P}(g_{2n}\in C_G(s^2)) \geq \left(\frac{3-\sqrt{5}-4\alpha}{2}\right)^2.\]
By \eqref{eq:proba_of_coset}, the left-hand side is equal to $\frac{1}{|G:C_G(s^2)|}$ and the claim is proven.
\end{proof}
\begin{proof}[Proof of \ref{item:involution_implies_virt_ab}. in Lemma \ref{lemma:Mikael2}] Let $G,s,m,H_i,a_i$ satisfying \eqref{eq:main_equation}, with $s^2=1$ and $\alpha<\frac{(1-\alpha)^3}{24}$.

Our main goal will be to prove that two randomly chosen elements of $G$ (for well-chosen probability measures on $G$ that are not exactly random walks but mixtures of random walks) commute with non-vanishing probability. By \cite{tointon}, we will deduce that $G$ is \emph{virtually abelian}. We proceed by contradiction, and assume that $G$ is not virtually abelian.

The element $s$ having order $2$, we can as well assume that the probability measure $\mu$ is $s$-left-invariant, that is it satisfies $\mu(sg) = \mu(g)$ for every $g \in G$. By~\eqref{eq:main_equation} and \eqref{eq:proba_of_coset}, we know that
\[ \liminf_n \mathbf{P}(g_n \in \sq^{-1}(1) \cup s\sq^{-1}(1))\geq 1-\alpha.\]
By the $s$-invariance of $\mu$, $\mathbf{P}(g_n \in \sq^{-1}(1)) = \mathbf{P}(g_n \in s\sq^{-1}(1))$, so $\liminf_n \mathbf{P}(g_n^2=1) \geq \frac{1-\alpha}{2}$.

As before, since $g^2=h^2 = (gh)^2$ implies that $g$ and $h$ commute, \eqref{eq:proba_of_commuting} implies that
\[\lim_n \mathbf{P}(g_n^2 = {g'_n}^2 = (g_n g'_n)^2 = 1) = 0.\]
On the other hand, $g_ng'_n$ being distributed as $g_{2n}$, we have
\[ \liminf_n \mathbf{P}( (g_ng'_n)^2=1 \textnormal{ or }  (sg_ng'_n)^2=1)\geq 1-\alpha.\]
Let now $g_n^{(1)}, g_n^{(2)}$ and $g_n^{(3)}$ be three independent copies of the random walk on $G$ given by $\mu$. Let $A_n$ be the event
\[ A_n = \{  \forall 1 \leq  i\neq j \leq 3,  (g_n^{(i)})^2= 1 \textnormal{ and } (sg_n^{(i)}g_n^{(j)})^2=1\}.\] It follows from the preceding discussion that the difference of $\{ \forall 1 \leq  i \leq 3,  (g_n^{(i)})^2= 1\}$ 
and $A_n$ has probability $\leq 3\alpha +o(1)$, so 
\[\liminf_n \mathbf{P}(A_n) \geq \liminf_n \mathbf{P}( \forall 1 \leq  i \leq 3, (g_n^{(i)})^2= 1)-3\alpha \geq \frac{(1-\alpha)^3}{8} - 3\alpha.\]
This is strictly positive by assumption.

But on $A_n$, for every $1 \leq i,j \leq 3$, we have
\[ s g_n^{(i)} g_n^{(j)} s^{-1} = (g_n^{(i)}g_n^{(j)})^{-1} = g_n^{(j)} g_n^{(i)}\]
and therefore
\[ (sg_n^{(1)}) g_n^{(2)} g_n^{(3)} (sg_n^{(1)})^{-1}
   = s g_n^{(1)} g_n^{(2)} s^{-1} s g_n^{(3)} g_n^{(1)} s^{-1}
   = g_n^{(2)} g_n^{(1)} g_n^{(1)} g_n^{(3)}
   = g_n^{(2)} g_n^{(3)}.\]
To say it differently, $s g_n^{(1)}$ commutes with $g_n^{(2)} g_n^{(3)}$. We deduce
\[\mathbf{P}( [s g_n^{(1)},g_n^{(2)}g_n^{(3)}]=1)\geq \mathbf{P}(A_n).\]
Using that $\mu$ is $s$-invariant, $sg_n^{(1)}$ is distributed as $g_n^{(1)}$ and \[\mathbf{P}( [g_n^{(1)},g_n^{(2)}g_n^{(3)}]=1) = \mathbf{P}( [sg_n^{(1)},g_n^{(2)}g_n^{(3)}]=1).\]
We can rewrite this as
\[ \liminf_n \mu^{\ast n} \otimes \mu^{\ast 2n}(\{(g,h) \in G \times G \mid [g,h]=1\}) >0.\]
Denote by $\nu_n$ the probability $\frac{1}{2}(\mu^{\ast n} + \mu^{\ast 2n})$. We clearly have $\nu_n \otimes \nu_n \geq \frac 1 4 (\mu^{\ast n} \otimes \mu^{\ast 2n})$, so
\[ \liminf_n \nu_n\otimes \nu_n ( \{(g,h) \in G \mid [g,h] = 1) >0.\]
On the other hand, it follows from \cite[Theorem 1.11]{tointon} that the sequences of measures $\mu^{\ast n}$ and $\mu^{\ast 2 n}$ (and therefore also $\nu_n$) measure index uniformly, hence the preceding is a contradiction with \cite[Theorem 1.9]{tointon}. So our starting assumption that $G$ is not virtually abelian is absurd. This concludes the proof of the lemma.
\end{proof}

\subsection{Counting triangles}
We now state and prove a lemma on the augmentation of the number of triangles (in Cayley graphs of $G$) containing some $s_0\in G$. 
It complements results of \cite[Lemma 9.2]{delaSalleTessera} and \cite[Lemma 31]{LdlS2020} which where valid for groups with elements of infinite (respectively very large) order. The conclusion of Lemma~\ref{Lemma:9.2} is also cleaner than in \cite{delaSalleTessera,LdlS2020}.
\begin{lemm}\label{Lemma:9.2}
Let $G$ be a finitely generated group that is not virtually abelian, and let $S\subset G\setminus \{1\}$ be a finite symmetric generating set.

Then for each $s_0$ in $S$, there exists $S\subset S'\subset G$ a finite symmetric generating set such that
\begin{enumerate}\renewcommand{\theenumi}{\alph{enumi}}
\item $\Delta\coloneqq S'\setminus S$ has at most $4$ elements;\label{ConditionA}
\item $\Delta$ does not contain elements of order at most $2$;\label{ConditionAA}
\item $\Delta\cap\setst{s^2}{s\in S}=\emptyset$;\label{ConditionB}
\item $\Triangles{s}{S'}\leq 6$ for all $s\in\Delta$;\label{ConditionC}
\item $\Triangles{s}{S'}=\Triangles{s}{S}$ for all $s\in S\setminus\{s_0,s_0^{-1}\}$;\label{ConditionD}
\item\label{item:cases_Lemma92} the value of $\Triangles{s_0}{S'}-\Triangles{s_0}{S}$ is equal to
\[\begin{cases}
	1 & \textnormal{if $s_0^2\neq 1$ and $C_G(s_0)$ is locally finite,}\\
	2 & \textnormal{if $s_0^2=1$ and $C_G(s_0)$ is locally finite,}\\
	2 & \textnormal{if $s_0^2\neq 1$ and $C_G(s_0)$ is not locally finite,}\\
	4 & \textnormal{if $s_0^2=1$ and $C_G(s_0)$ is not locally finite.}
\end{cases}\]
\end{enumerate}
\end{lemm}
\begin{proof}
Let $g$ be an element of $G$ and $\Delta_g\coloneqq\{g,g^{-1},s_0^{-1}g,g^{-1}s_0\}$ .
We will show that there exists some $g$ in $G$ such that $S'=S'_g\coloneqq S\cup\Delta_g$ works.
Observe that for all $g$, the set $S'_g$ satisfies Condition \ref{ConditionA} of the lemma, and that it satisfies \ref{ConditionAA} if and only if $g^2 \neq 1$ and $(s_0^{-1}g)^2 \neq 1$, or equivalently $g \notin \sq^{-1}(1) \cup s_0 \sq^{-1}(1)$.

We first restrict our attention to elements $g$ such that the following two conditions hold
\begin{gather}
	\abs{g}_S\geq 3\label{Condition1}\\
	\abs{s_0^{-1}g}_S\geq 3\label{Condition2}
\end{gather}
where $\abs g_S$ is the word length of $g$ relative to the generating set $S$.

Since $S$ is finite, the number of $g\in G$ such that one of the conditions \eqref{Condition1}-\eqref{Condition2} do not hold is finite.
Also, for a $g$ satisfying Conditions \eqref{Condition1} and \eqref{Condition2} the intersection $\Delta_g\cap S$ is empty and Condition \ref{ConditionB} is automatically satisfied.
Moreover, in the Cayley graph of $G$ relative to $S'_g$, a triangle with a side labelled by  $s\in\Delta_g$ has at least another side  labelled by an element of $\Delta_g$, otherwise $s$ would have $S$-length at most $2$.
This implies that any edge labelled by $s\in\Delta_g$ belongs to at most $6$ triangles in $\unCayley{G}{S'}$, which is Condition \ref{ConditionC}.
Indeed, if one edge $e$ is labelled by $s\in\Delta_g$, there are at most three possibilities to put an edge labelled by $t\in \Delta_g\setminus\{s^{-1}\}$ at each extremity of $e$, thus giving a maximum number of $2\cdot 3=6$ triangles containing~$e$.
This also shows that for any $s\in S$ we have 
\[
	\Triangles{s}{S'_g}-\Triangles{s}{S}=\abs{\setst{t\in\Delta_g}{s^{-1}t \in \Delta_g}}=\abs{\Delta_g\cap s\Delta_g}
\]

We now turn our attention on the set $\Delta_g\cap s\Delta_g$.
Its cardinality is equal to the number of pairs $(u,v) \in \Delta_g$ such that $u=sv$.
By replacing $u$ and $v$ by the words $g,g^{-1},s_0^{-1}g$ and $g^{-1}s_0$, this gives us $16$ equations in the group.
Among these $16$ equations, $4$ imply that $s=1$.
The $12$ remaining equations for elements of $\Delta_g\cap s\Delta_g$ are shown in Table \ref{TableDeltag}.
\begin{table}
\[\arraycolsep=1.6pt\def\arraystretch{1.4}\begin{array}{|l||l|}
	\hline
	\textnormal{Possible elements of }\Delta_g\cap s\Delta_g & \textnormal{Occurs if}\\
	\hline\hline
	g=ss_0^{-1}g & s=s_0\\
	\hline
	s_0^{-1}g=sg & s=s_0^{-1}\\
	\hline
	g^{-1}s_0=sg^{-1} & s=g^{-1}s_0g\\
	\hline
	g^{-1}=sg^{-1}s_0 & s=g^{-1}s_0^{-1}g \\
	\hline
	g=sg^{-1} &\sq(g)=s\\
	\hline
	g^{-1}=sg &\sq(g)=s^{-1}\\
	\hline
	s_0^{-1}g=sg^{-1} &\sq(g)=s_0s\\
	\hline
	g^{-1}=ss_0^{-1}g &\sq(g)=s_0s^{-1}\\
	\hline
	g=sg^{-1}s_0 &\sq(s_0^{-1}g)=s_0^{-1}s\\
	\hline
	g^{-1}s_0=sg &\sq(s_0^{-1}g)=s_0^{-1}s^{-1}\\
	\hline
	s_0^{-1}g=sg^{-1}s_0 &\sq(s_0^{-1}g)=s\\
	\hline
	g^{-1}s_0=ss_0^{-1}g &\sq(s_0^{-1}g)=s^{-1}\\
	\hline
\end{array}\]
\caption{Possible elements of $\Delta_g\cap s\Delta_g$, where $\sq(g)= g^2$.}
\label{TableDeltag}
\end{table}

In particular, if $g$ is as in the conclusion of Proposition~\ref{Prop:ultimate} for $F = S \cup s_0 S \cup s_0^{-1} S$ and $s=s_0$, we see that only the first two lines in this Table occur if $C_G(s_0)$ is locally finite, and only the first four  occur otherwise. This implies Condition~\ref{ConditionD}. Also, Condition \ref{ConditionAA} holds in this case since $1$ belongs to $s_0^{-1}S\subseteq F$, and Condition \ref{item:cases_Lemma92} is automatically satisfied.
The fact that there are infinitely many $g$ in the conclusion of Proposition~\ref{Prop:ultimate} imply that we can find such $g$ satisfying also conditions \eqref{Condition1}-\eqref{Condition2}.
\end{proof}

We are now ready to prove Proposition~\ref{Prop:9.3}.
\begin{proof}[Proof of Proposition~\ref{Prop:9.3}]The proof will be by successive applications of Lemma \ref{Lemma:9.2}. To prove the proposition, it is enough that all elements of $S$ belong to at least $7$ $\tilde S$-triangles (to distinguish them from the newly added elements which will belong to at most $6$ $\tilde S$-triangles) and that the numbers $\Triangles{s^{\pm1}}{\tilde S}$ for $s$ in $S$ are all distinct.

Let $S_0 = S \cup S^{-1}$, so that $|S_0|\leq 2|S|$. Let $s_1,\dots,s_n$ be any enumeration of the elements of $S$. Apply successively Lemma \ref{Lemma:9.2} at most $7$ times with $s_1$, to get a set $S_1$ containing $S$ and such that $\Triangles{s_1}{S_1}$ is larger than $7$. Then, applying Lemma \ref{Lemma:9.2} for $s_2$ at most $8$ times, we can bring $\Triangles{s_2}{S_2}$ to another value $\geq 7$. Doing the same for each element of $S$, we finally obtain a set $\tilde S$ as in the lemma, after a total number of $\leq 7+8+ \dots + (\abs S+6)=\abs S(\abs S+13)/2$ successive applications of Lemma \ref{Lemma:9.2}. At the end, we have
\[ |\tilde S| \leq |S_0|+4 \frac{\abs S(\abs S+13)}{2} \leq 2\abs S(\abs S+14).\qed\]
\let\qed\relax
\end{proof}

\section{Proofs of the main results}\label{sec:conclusion}
We collect here for completeness the straightforward proofs of the results from the introduction.
\begin{proof}[Proof of Theorem~\ref{thm:main}] Let $G$ be as in Theorem~\ref{thm:main}, and $S_0$ be a finite generating set. Let $S_1 = (S_0 \cup S_0^2 \cup S_0^3)\setminus\{1\}$, and $S_2$ be the generating set $\tilde S$ given by Proposition~\ref{Prop:9.3} for $S=S_1$. By Proposition~\ref{Prop:9.3} and the discussion preceding it, every automorphism of $\unCayley{G}{S_2}$ preserves the $S_1$-colours: $\phi(gs) \in \{\phi(g) s,\phi(g) s^{-1}\}$ for every $g \in G$ and $s \in S_1$. By Theorem~\ref{orientation-rigid}, $\phi$ is a left-translation by an element of $G$. 
\end{proof}

\begin{proof}[Proof of Corollary~\ref{cor:main}] If $G$ is finite, the equivalence is the content of \cite{MR642043}. We can assume that $G$ is infinite and finitely generated. The implication \ref{item:locFiniteGRR} $\implies$ \ref{item:GRR} is obvious, and the implication \ref{item:GRR}$\implies$\ref{item:Gnonexceptionnal} is known and very easy, see~\cite{MR0280416}. For the reader's convenience, we recall the argument. If an infinite finitely generated group $G$ is either abelian or generalized dicyclic then there is a nontrivial permutation $\varphi$ of $G$ satisfying $\varphi(gh) \in \{\varphi(g)h,\varphi(g)h^{-1}$ for every $g,h \in G$: take for $\varphi$ the inverse map if $G$ is abelian, and the map that is the identity on $A$ and the inverse on $G \setminus A$ if $G$ is generalized dicyclic and $x,A$ are as in Definition~\ref{def:gen_dicyclic}). In particular, $\varphi$ induces an automorphism of every Cayley graph of $G$, different from a translation. Observe that this argument even rules out the existence of a non-locally finite \GRR{}.
  
We have to justify \ref{item:Gnonexceptionnal} $\implies$ \ref{item:locFiniteGRR}. If $G$ is not virtually abelian, then \ref{item:locFiniteGRR} is the conclusion of Theorem~\ref{thm:main}. Otherwise, $G$ admits an element of infinite order (a torsion abelian finitely generated group is finite), and \cite[Theorem 2]{LdlS2020} applies and proves~\ref{item:locFiniteGRR}.
\end{proof}
\begin{proof}[Proof of Corollary~\ref{cor:discrete_automorphism_group}] If $G$ is not virtually abelian, this is a particular case of Theorem~\ref{thm:main}. Otherwise, as explained in the proof of Corollary~\ref{cor:main}, $G$ has an element of infinite order and \cite[Theorem J]{delaSalleTessera} applies.
\end{proof}

\begin{proof}[Proof of Corollary~\ref{cor:LGrigid}]
  Combine Corollary~\ref{cor:discrete_automorphism_group} with \cite[Theorem E]{delaSalleTessera}.
\end{proof}

\section{A conjecture on the squares of a random walk}\label{sec:conjecture}

We mentionned before Lemma~\ref{lemma:prob_of_involution} that we expect that the following conjecture holds.
\begin{conj}\label{conj:proba_of_sq} Let $G$ be a finitely generated group that is not virtually abelian, $\mu$ be a symmetric probability measure on $G$ with finite and generating support containing the identity, and $(g_n)$ a realization of the random walk on $G$ given by $\mu$. Then
  \begin{equation}\label{eq:proba_of_square} \forall a \in G, \lim_n \mathbf{P}(g_n^2=a) = 0.\end{equation}
\end{conj}

Better, there should be a function $f:(0,1] \to \mathbf{N}$ such that, in a group $G$, if
\[ \exists a \in G, \limsup_n \mathbf{P}(g_n^2=a) \geq \varepsilon,\]
then $G$ admits an abelian subgroup of index $\leq f(\varepsilon)$. This is known to be true for finite groups \cite{MR1242094,MR3899225}.

The main case of the conjecture is when $a=1$: indeed with similar methods as the reduction to $F=\{1\}$ in the proof of Proposition \ref{Prop:ultimate}, one can show that the case $a=1$ in Conjecture~\ref{conj:proba_of_sq} implies the full conjecture for $G$ not virtually $2$-nilpotent.

Let us mention here that this conjecture would allow to greatly simplify our proofs, as it would imply immediately the following variant of Proposition \ref{Prop:ultimate}, which also implies the main Theorem~\ref{thm:main} by the same argument.
\begin{lemm} If Conjecture~\ref{conj:proba_of_sq} holds for $G$, then for every $s \in G$ and $F \subset G$ finite there are infinitely many $g \in G \setminus (\sq^{-1}(F) \cup s \sq^{-1}(F))$ such that
\[\begin{cases}
g^{-1} s g \notin F&\textnormal{if $C_G(s)$ has infinite index in $G$}\\
g^{-1} s g = s&\textnormal{otherwise.}
\end{cases}\]
\end{lemm}
\begin{proof}
We prove the stronger fact that the probability that $g=g_n$ satisfies the conclusion of the lemma is $1-o(1)$ when $|G:C_G(s)|=\infty$, and $\frac{1}{|G:C_G(s)|} - o(1)$ otherwise.
  
  It follows from \eqref{eq:proba_of_square} that
  \[ \lim_n \mathbf{P}(g_n \in \sq^{-1}(F)) = 0.\]
It also implies that
\begin{equation}\label{eq:proba_of_square3} \lim_n \mathbf{P}(g_n \in s \sq^{-1}(F)) = 0.\end{equation}
To justify this, we need to introduce an independant copy $(g'_n)_{n \geq 0}$ of the random walk $(g_n)$. Since the support of $\mu$ is symmetric  and generates $G$, there is a $k$ such that $\mathbf{P}(g'_k=s^{-1})>0$. So using that $g_{n+k}$ is distributed as $g'_k g_n$, we obtain
\begin{align*} \mathbf{P}(g_{n+k} \in \sq^{-1}(F)) &\geq \mathbf{P}(s^{-1}g_n \in \sq^{-1}(F) \textrm{ and }g'_k = s^{-1}) \\
  & = \mathbf{P}(g_n \in s \sq^{-1}(F)) \mathbf{P}(g'_k = s^{-1}).
\end{align*}
This proves \eqref{eq:proba_of_square3}.
Moreover, it follows from \eqref{eq:proba_of_coset} that
\[\begin{cases}
\lim_n\mathbf{P}(g_n^{-1} s g_n \notin F) = 1&\textnormal{if $|G:C_G(s)|=\infty$}\\
\lim_n\mathbf{P}(g_n^{-1} s g_n =s) = \frac{1}{|G:C_G(s)|}&\textnormal{otherwise.}
\end{cases}\]
The conclusion follows.
\end{proof}

\section{Directed and oriented graphs}\label{sec:directed}
A natural variation of Cayley graphs is the concept of \emph{Cayley digraph} (directed graph). 
Given a group $G$ and a (not necessarily symmetric) generating set $S\subseteq G\setminus\{1\}$, the Cayley digraph $\orCayley{G}{S}$ is the digraph with vertex set $G$ and with an arc (directed edge) from $g$ to $h$ if and only if $g^{-1}h\in S$.

A Cayley digraph $\orCayley{G}{S}$ of $G$ whose automorphism group acts freely on its vertex set is called a \emph{digraphical regular representation}, or \DRR.
If moreover $\orCayley{G}{S}$ has no bigons (that is if $S\cap S^{-1}=\emptyset$) then we speak of an \emph{oriented graphical regular representation}, or \ORR.

We have the directed equivalent of Corollary~\ref{cor:main}:

\begin{prop}\label{prop:DRR} For a finitely generated group $G$, the following are equivalent:
  \begin{enumerate}
  \item\label{item:DRR} $G$ admits a \DRR,
  \item\label{item:locFiniteDRR} $G$ admits a finite degree \DRR,
  \item\label{item:GnonDexceptionnal} $G$ is neither the quaternion group $Q_8$, not any of $(\Z/2\Z)^2$, $(\Z/2\Z)^3$, $(\Z/2\Z)^4$, $(\Z/3\Z)^2$.
  \end{enumerate}
\end{prop}
\begin{proof}
If $G$ is finite, the equivalence is the content of \cite{MR603394}. We can assume that $G$ is infinite and finitely generated. The implication \ref{item:locFiniteDRR} $\implies$ \ref{item:DRR} is obvious, and the implication \ref{item:DRR}$\implies$\ref{item:GnonDexceptionnal} is empty for infinite groups.
We have to justify \ref{item:GnonDexceptionnal} $\implies$ \ref{item:locFiniteDRR}.

Let $S$ be a finite generating set of $G$.
Using Proposition~\ref{Prop:9.3} and \cite[Lemma~32]{LdlS2020} we obtain a finite generating set $S\subseteq \tilde S\subset G$ such that for all $s\in S$ and $t\in\tilde S$, if $\Triangles{s}{\tilde S}=\Triangles{t}{\tilde S}$ then $t=s$ or $t=s^{-1}$.
By \cite[Lemma 5]{LdlS2020} and \cite[Proposition 6]{LdlS2020} there exists a generating set $T\subseteq\tilde S$ such that $\orCayley{G}{T}$ is a \DRR. Moreover, $T\cap T^{-1}$ consist only of elements of order $2$.
\end{proof}

Observe that the equivalence of \ref{item:DRR} and \ref{item:GnonDexceptionnal} was the content of \cite{MR0498225,MR603394}.

We will conclude with the oriented equivalent of Corollary~\ref{cor:main} and thus answer \cite[Problem 2.7]{MR603394}.
Recall that a generalized dihedral group $G$ is the semi-direct product $A\rtimes \Z/2\Z$ where $A$ is an abelian group and $\Z/2\Z$ acts on $A$ by inversion.

\begin{prop} For a finitely generated group $G$, the following are equivalent:
\begin{enumerate}
  \item\label{item:ORR} $G$ admits an \ORR,
  \item\label{item:locFiniteORR} $G$ admits a finite degree \ORR,
  \item\label{item:GnonOexceptionnal} $G$ does not belong to the following list:
  \begin{itemize}
    \item the non-trivial generalized dihedral groups,
    \item the following $11$ finite groups of cardinality at most $64$: $Q_8$, $\Z/4\Z\times\Z/2\Z$, $\Z/4\Z\times(\Z/2\Z)^2$, $\Z/4\Z\times(\Z/2\Z)^3$, $\Z/4\Z\times(\Z/2\Z)^4$, $(\Z/3\Z)^2$, $\Z/3\Z\times(\Z/2\Z)^3$, $D_4 \circ D_4$ (the central product of two dihedral groups of order $8$, which has order $32$) and the three groups (of respective orders $16$, $16$ and $32$) given by the presentations
    \[ \presentation{a,b}{a^4=b^4=(ab)^2=(ab^{-1})^2=1},\]
    \begin{gather*} \langle a,b,c \,|\, a^4=b^4=c^4=(ba)^2=(ba^{-1})^2=(bc)^2=(bc^{-1})^2=1\\
    a^2c^{-2}=a^2b^{-2}=cac^{-1}a^{-1}=1\rangle,
    \end{gather*}
    \begin{gather*} \langle a,b,c\,|\, a^4=b^4=c^4=(ab)^2=(ab^{-1})^2=(ac)^2=(ac^{-1})^2=1\\
    (bc)^2=(bc^{-1})^2=a^2b^2c^2=1\rangle.
    \end{gather*}
  \end{itemize}
\end{enumerate}
\end{prop}
\begin{proof}
If $G$ is finite, the equivalence is the content of \cite{MR3873496}, while every generating set of a generalized dihedral group contains an element of order $2$ (namely any element not in $A$). 
Once again, we have to justify \ref{item:GnonOexceptionnal} $\implies$ \ref{item:locFiniteORR} for $G$ infinite.

Let $G$ be a finitely generated group which is not generalized dihedral.
Then by \cite[Proposition 5.2]{MR0498225} there exists a finite generating set $S$ of $G$ without elements of order~$2$.
Then the generating set $\tilde S$ given by Proposition~\ref{Prop:9.3} and \cite[Lemma~32]{LdlS2020} has also no elements of order $2$.
This implies that for $T$ given by \cite[Lemma 5]{LdlS2020} and \cite[Proposition 6]{LdlS2020} the \DRR{} $\orCayley{G}{T}$ is actually an \ORR. 
\end{proof}

\bibliographystyle{plain}
\bibliography{../Biblio}
\providecommand{\noopsort}[1]{} \def\cprime{$'$}

\end{document}